\documentclass[a4paper, 11pt]{amsart}

\usepackage{amsmath, amsfonts, amssymb}
\usepackage{amsthm}
\usepackage{esint}
\usepackage{tikz}
\usetikzlibrary{matrix}
\usepackage[
bookmarks=false]{hyperref}
\usepackage{cite}
\usepackage{array}
\usepackage [english]{babel}
\usepackage [autostyle, english = american]{csquotes}
\MakeOuterQuote{"}

\usepackage{amsaddr}
\usepackage{enumerate}
\usepackage{pgfplots}

\title{On Generalized Dirichlet Integrals in the Smooth and in the O-minimal Setting}

\author{Gian Maria Dall'Ara}
\address{Istituto Nazionale di Alta Matematica ``F. Severi"\\ Research Unit Scuola Normale Superiore\\
	Piazza dei Cavalieri, 7, 56126, Pisa (Italy)}
\email{dallara@altamatematica.it}

\thanks{2020 Mathematics Subject Classification:}

\date{\today}

\newcommand{\R}{\mathbb{R}}

\newtheorem{thm}{Theorem}
\newtheorem{prp}[thm]{Proposition}
\newtheorem{lem}[thm]{Lemma}
\newtheorem{cor}[thm]{Corollary}

\newtheorem{dfn}{Definition}
\newtheorem{ex}{Example}

\begin{document}

\maketitle

\tableofcontents

\begin{abstract}
Given a compact manifold $M$ equipped with smooth vector fields $X_1,\ldots, X_r$, we consider the generalized Dirichlet energy \[\mathbf{E}(f)= \sum_{j=1}^r\int_M |X_jf|^2\, dm,\] where $dm$ is a volume form, and ask if the set \[
\mathcal{B}=\{f\in L^2(M)\colon\,\mathbf{E}(f)+\lVert f\rVert_{L^2(M)}^2\leq 1 \}
\] is precompact in $L^2(M)$. We find a geometric sufficient condition in terms of "iterated characteristic sets" and use it to show that, if the vector fields are tame (in the sense of o-minimality) and satisfy the Hörmander condition of some order $s$ on a dense set of points, then the only obstruction to precompactness is the existence of a characteristic submanifold (i.e.~a nonempty submanifold of positive codimension to which each $X_j$ is tangent). Implications for global regularity of sum-of-squares operators not necessarily satisfying Hörmander condition are discussed in an appendix.
	\end{abstract}

\section{Introduction}

Let $M$ be a smooth compact manifold without boundary and let $X_1,\ldots, X_r$ be a finite collection of smooth vector fields, globally defined on $M$. If $dm$ is a smooth volume form on $M$, we consider the \emph{generalized Dirichlet energy}  \begin{equation}\label{eq:dirichlet_energy}
\mathbf{E}(f)= \sum_{j=1}^r\int_M |X_jf|^2\, dm.
\end{equation}
This is well-defined for every $f\in L^2(M)$ with the property that the distributional derivative $X_jf$ is in $L^2(M)$ for every $j$.

The basic question investigated in this paper is the following: \emph{under which assumptions on $X_1,\ldots, X_r$ is the set \[
\mathcal{B}=\{f\in L^2(M)\colon\, X_jf\in L^2(M)\quad \forall j\quad \&\quad \mathbf{E}(f)+\lVert f\rVert_{L^2(M)}^2\leq 1\}
\]
precompact in $L^2(M)$?} Here $\lVert f\rVert_{L^2(M)}^2=\int_M|f|^2\, dm$. We will refer to $\mathcal{B}$ as the \emph{Dirichlet energy ball} associated with the given vector fields.

Precompactness of $\mathcal{B}$ may be thought of as a form of the so-called Dirichlet principle. Notice that its validity does not depend on the choice of volume form. By a classical theorem of Kohn and Nirenberg (Theorem 3 in \cite{kohn_nirenberg}), it implies smoothness of global (weak) solutions of an associated sum-of-squares operator. This is explained in some detail in Appendix \ref{sec:hypoellipticity}. We remark that global hypoellipticity of sum-of-squares and related operators has been investigated by several authors. See, e.g., \cite{amano, oleinik_radkevic,greenfield_wallach, fujiwara_omori, hounie ,gramchev_popivanov_yoshino, bergamasco_cordaro_malagutti, omori_kobayashi, himonas_petronilho, himonas_petronilho_dossantos, araujo_ferra_ragognette}.  \newline 

The content of the paper is the following. In Section \ref{sec:weierstrass_hormander} we recall the well-known rôle played in our problem by characteristic submanifolds (for the given family of vector fields) and by the Hörmander condition on iterated commutators. A proof of a generalized Weierstrass counterexample presented in this section is provided, for the sake of completeness, in Appendix \ref{sec:weierstrass}. 

In Section \ref{sec:degeneration}, the spotlight is on the \emph{degeneration locus}, that is, the set of points where Hörmander condition fails. Our answer to the precompactness question above depends on the structure of this set, and more precisely of its descending chain of \emph{iterated characteristic subsets} built by iteratively removing the non-characteristic points. We state a sufficient condition for precompactness (Theorem \ref{thm:sufficient}), whose proof is deferred to Section \ref{sec:sufficient}. Appendix \ref{sec:hypoellipticity} contains some remarks on the relation of Theorem \ref{thm:sufficient} to a theorem of K. Amano \cite{amano}.

In Section \ref{sec:o_minimal} we restrict our considerations to tame vector fields, that is, vector fields not displaying any of the "pathologies" that may appear in the smooth category. To formalize the notion of "non-pathological", we use o-minimality, a theory originating in model theory that has recently served a similar purpose in various settings (e.g., Diophantine \cite{pila_icm}, algebraic \cite{fresan} and complex analytic geometry \cite{peterzil_starchenko}). The idea is that, by restricting consideration to the o-minimal setting, we may achieve a neat geometric characterization of precompactness of the Dirichlet energy ball in a significant generality (Theorem \ref{thm:tame}), going well-beyond the real-analytic setting. The proof of Theorem \ref{thm:tame} occupies Sections \ref{sec:proof_tame_1}, \ref{sec:proof_tame_2}, and \ref{sec:proof_tame_3}. The key to the proof is a finite stabilization lemma (Lemma \ref{lem:finite_termination}) for the iterated characteristic sets introduced in Section \ref{sec:degeneration}, which is valid in the o-minimal setting. \newline 

This note has been significantly influenced by the ideas of various people. The iterated characteristic sets considered here are a simpler analogue of the iterated derived distributions of my paper \cite{dallara_mongodi} with S. Mongodi. The idea of proving a finite stabilization lemma in the o-minimal setting has been given to me by B. Lamel at the 2022 iteration of the Levico conference "CR geometry and PDEs". In a forthcoming joint paper with A. Berarducci, B. Lamel, M. Mamino and S. Mongodi, we work it out in the setting of the $\overline\partial$-Neumann problem. The proof of Lemma \ref{lem:finite_termination} below is essentially a simpler analogue of the proof of finite stabilization in that paper. Finally, the arguments in Section \ref{sec:sufficient} are inspired by a conversation I had with E. Straube around his paper \cite{straube}. 

\section{Weierstrass counterexample and Hörmander condition}\label{sec:weierstrass_hormander}

It was Weierstrass, in his famous criticism of Riemann's liberal usage of the Dirichlet principle, who first noticed that $\mathcal{B}$ need not be precompact in the space of square-integrable functions. More precisely, \cite[pp. 53-54]{weierstrass} essentially shows that precompactness fails in the one-dimensional case if there is a point at which all the vector fields $X_1,\ldots, X_r$ vanish. We now state a (well-known) generalization of the Weierstrass counterexample to higher dimensions, which gives us the opportunity of recalling a notion that is crucial for what follows. 

\begin{dfn} A characteristic submanifold, with respect to the vector fields $X_1,\ldots, X_r$, is a non-empty submanifold $N\subset M$ of positive codimension with the property that each $X_j$ is tangent to $N$ at each of its points. 
	\end{dfn}

\begin{prp}[Generalized Weierstrass Counterexample]\label{prp:weierstrass} Assume that there is a characteristic submanifold with respect to the given vector fields. Then the associated Dirichlet energy ball is not precompact in $L^2(M)$.	
	\end{prp}

The proof is a standard scaling argument recalled in Appendix \ref{sec:weierstrass}.\newline 

Of course, the Weierstrass counterexample does not apply to the elliptic Dirichlet integral that was originally of interest to Riemann. Indeed, if the vector fields $X_1,\ldots, X_r$ span the tangent space at each point of $M$, then $\mathcal{B}$ is the unit ball (for an appropriate choice of equivalent norm) in the $L^2$-Sobolev space of order $1$ on $M$, which is precompact in $L^2(M)$ by a well-known compact embedding theorem. Hörmander's famous work on sum-of-squares operators gives a remarkable generalization of this fact: if the \emph{vector fields $X_1,\ldots, X_r$ together with their iterated commutators of arbitrary length} \[
[X_{i_1}, \ldots, [X_{i_{k-1}}, X_{i_k}]\ldots]\qquad (1\leq i_1, \ldots, i_k\leq r, \, k\geq 2)
\]\emph{span the tangent space to $M$ at every point}, then we have the subelliptic estimate (cf.~estimate (3.4) in \cite{hormander}) \begin{equation}\label{eq:subelliptic}
\lVert f\rVert_{W^s(M)}^2\leq C\left(\mathbf{E}(f)+\lVert f\rVert_{L^2(M)}^2\right)\qquad\forall f\in C^\infty(M).
\end{equation}
Here $W^s(M)$ is the $L^2$-Sobolev space of order $s$, $s$ is some small positive number (depending on the length of the commutators needed to span the tangent space at every point), and the constant is of course independent of $f$. This subelliptic estimate and the compact embedding $W^s(M)\hookrightarrow L^2(M)$ readily show that $\mathcal{B}\cap C^\infty(M)$ is precompact in $L^2(M)$ whenever the above \emph{Hörmander condition} on iterated commutators holds true. Since $\mathcal{B}\cap C^\infty(M)$ is $L^2$-dense in $\mathcal{B}$ by the classical Friedrichs Lemma (see, e.g., \cite[Lemma 25.4]{treves}), we have precompactness of $\mathcal{B}$ in this case.\newline 

The existence of a characteristic submanifold $N$ is an obstruction to the validity of the Hörmander condition, as any iterated commutator of the given vector fields must necessarily be tangent to $N$ at each of its points. On the other hand, the Hörmander condition may fail at some point without any characteristic submanifold passing through that point. Among the simplest examples, one may consider the vector fields $X_1=\partial_x$, $X_2=g(x)\partial_y$ on the $2$-torus $\mathbb{T}^2_{x,y}$, where $g$ vanishes to infinite order at a single point $x_*\in \mathbb{T}_x$. In this case, Hörmander condition trivially holds off the circle $\{x=x_*\}$, and fails at every point of the circle, because the non-zero iterated commutators of $X_1$ and $X_2$ are all of the form $\pm g^{(k)}(x)\partial_y$ ($k\geq 0$). Yet, the circle is transversal to $X_1$ and hence not characteristic. This favourable feature is responsible for the fact that the Dirichlet energy ball is precompact in this case, despite the failure of Hörmander condition (this follows from Theorem \ref{thm:sufficient} below). 

\section{Degeneration locus and a sufficient condition}\label{sec:degeneration}

To put our hands on vector fields not covered by Hörmander's result, it is natural to turn our attention to the \emph{degeneration locus}
\begin{eqnarray*}
Z&:=&\{p\in M\colon\, \text{Hörmander condition does not hold at $p$}\}\\
&=&\{p\in M\colon\,\mathrm{span}_{1\leq i_1,\ldots, i_k\leq r,\, k\geq 1}\{[X_{i_1}, \ldots, [X_{i_{k-1}}, X_{i_k}]\ldots](p)\}\neq T_pM\}, 
\end{eqnarray*}
which is a closed subset of $M$. This is the set of "bad" points with respect to the given vector fields. As the next proposition shows, for us the interesting case is when $Z$ has empty interior. 

\begin{prp}\label{prp:large_degeneration_locus} 
	Suppose that the degeneration locus $Z$ has nonempty interior. Then an open subset of $Z$ is foliated by characteristic submanifolds, and therefore the Dirichlet energy ball is not precompact. 
	\end{prp}

This is an immediate corollary of Frobenius Theorem (e.g., one may apply \cite[Cor. 8.3.]{sussmann} to the open subset of $Z$ where the Lie algebra generated by the vector fields has maximal rank) and Proposition \ref{prp:weierstrass}. \newline 

As anticipated above, not all bad points are alike: characteristic submanifolds, which must be contained in $Z$, are really bad, while we will see that "non-characteristic points" are not obstructions to precompactness. In order to make the last statement precise, we need another definition, which allows us to talk about characteristic points of arbitrary subsets of $M$.  

\begin{dfn}\label{dfn:char}
Let $E$ be a subset of $M$. A point $p$ of $E$ is said to be characteristic, with respect to the given vector fields, if each $X_j$ is $C^1$-Zariski tangent to $E$ at $p$, and noncharacteristic otherwise. We denote by $\mathrm{char}(E)$ the set of characteristic points of $E$. 	
	\end{dfn}

Here we are using a specific notion of tangency of a vector to an arbitrary subset of a manifold, namely the following.

\begin{dfn} A vector $v\in T_pM$ is \emph{$C^1$-Zariski tangent} to $E$ if, for every $C^1$ function $f:M\rightarrow \R$ vanishing identically on $E$, we have $df_p(v)=0$.
	\end{dfn} 
It is easy to see that, if $E$ is closed, then $\mathrm{char}(E)$ is a closed subset of $E$. 

Whenever $\mathrm{char}(E)\subsetneq E$, tangency to the characteristic set is a more stringent condition than tangency to $E$, and $\mathrm{char}(\mathrm{char}(E))$ may happen to be smaller than $\mathrm{char}(E)$. Similarly, $\mathrm{char}(\mathrm{char}(\mathrm{char}(E)))$ may be even smaller etc. Thus, we are led to consider the \emph{chain of iterated characteristic sets of $E\subseteq M$}: \begin{equation}\label{eq:descending_chain}
E_0:=E \supseteq E_1\supseteq \ldots \supseteq E_k\supseteq \ldots, 
\end{equation}
where $E_{k+1}:=\mathrm{char}(E_k)$ for every $k\geq 0$. 

In particular, we may apply this construction to the degeneration locus $Z$. It is clear that any characteristic submanifold $N$ must be contained in $\cap_{k\geq 0}Z_k$. We observe that the notion of characteristic point and \eqref{eq:descending_chain} are well-defined for a collection of vector fields on an arbitrary (i.e., not necessarily compact) manifold. 

\begin{ex} Let us compute \eqref{eq:descending_chain} on three pairs of vector fields defined on $M=\R^2_{x,y}$. The reader may easily adapt these examples to the compact setting, e.g., on the $2$-torus. 
\begin{enumerate}
	\item If $X_1=e^{-\frac{1}{y^2}}\partial_x$ and $X_2=e^{-\frac{1}{x^2}}\partial_y$, then \[
	Z_0=\{x=0\}\cup \{y=0\}, \quad Z_1=\{0\}=Z_2=\ldots
	\]
\item If $X_1=\partial_x$ and $X_2=e^{-\frac{1}{x^2}}y\partial_y$, then \[
Z_0=\{x=0\}\cup \{y=0\}, \quad Z_1=\{y=0\}=Z_2=\ldots
\] 
\item  If $X_1=e^{-\frac{1}{y^2}}\partial_x$ and $X_2=\partial_x+e^{-\frac{1}{x^2}}y\partial_y$, then \[
Z_0=\{x=0\}\cup \{y=0\}, \quad Z_1=\{0\}, \quad Z_2=\emptyset=Z_3=\ldots
\]
	\end{enumerate}	
Notice that in all the above examples the chain stabilizes in finitely many steps to a characteristic submanifold or the empty set. 
	\end{ex}

We have the following theorem, whose proof is given in Section \ref{sec:sufficient}. 

\begin{thm}[Sufficient condition for precompactness of $\mathcal{B}$]\label{thm:sufficient}
Assume that there exists a natural number $k$ with the property that some iterated characteristic set $Z_k$ of the degeneration locus is empty. Then the Dirichlet energy ball $\mathcal{B}$ is precompact. 
	\end{thm}

See Appendix \ref{sec:hypoellipticity} for a few remarks on the implications of the above result for global hypoellipticity of sum-of-squares operators and for a comparison with a related one discovered by K.~Amano (Theorem 1 of \cite{amano}) in the context of hypoellipticity of second-order linear PDEs. \newline 

The hypothesis of Theorem \ref{thm:sufficient} may be viewed as a strenghtening of the condition "there are no characteristic submanifolds". In this sense, the generalized Weierstrass counterexample of Proposition \ref{prp:weierstrass} and the sufficient condition of Theorem \ref{thm:sufficient} complement each other. Yet, the two results fall short of giving a complete characterization of precompactness for arbitrary families of \emph{smooth} vector fields. This is due to the possibility of "wild" behaviors, ultimately stemming from the basic fact that zero sets of smooth functions are nothing better than arbitrary closed sets. More precisely: 
 \begin{enumerate}
\item \emph{The chain of iterated characteristic sets of the degeneration locus may stabilize after $k$ steps with $Z_k=Z_{k+1}=\ldots$ nonempty and not containing any characteristic submanifold}. For a simple instance of this phenomenon, consider $M=\mathbb{T}^2_{x,y}$ and $X_1=\partial_x$, $X_2=g(x)\partial_y$, with $g$ vanishing (necessarily to infinite order) exactly on a Cantor set (or, more generally, a perfect set). Every point of the degeneration locus $Z_0=\{x\colon\, g(x)=0\}\times \mathbb{T}_y$ is characteristic, and yet no proper submanifold of $Z_0$ is characteristic, as such a manifold would be contained in a vertical circle.
\item \emph{The chain of iterated characteristic sets of the degeneration locus may not stabilize at all}. Even more dramatically, the closed set $Z_\omega:=\cap_k Z_k$ may not be characteristic. This is seen by considering vector fields of the same form as in the previous item, taking this time for the zero set of $g$ a completely arbitrary closed set. We observe that if one defines $Z_\alpha$ for every ordinal $\alpha$ by transfinite recursion (letting $Z_\alpha=\cap_{\beta<\alpha}Z_\beta$ when $\alpha$ is a limit ordinal), then the Cantor--Bendixson theorem guarantees that this extended chain must eventually stabilize, i.e., $\mathrm{char}(Z_\alpha)=Z_\alpha$ for some countable ordinal. See \cite[Section 2]{dallara_mongodi} for the a detailed explaination of the same phenomenon in a slightly different context. As we have no good use for this observation for the problem at hand, we will not elaborate on it.  
\end{enumerate}

While these pathological behaviors seem to indicate that simple considerations on iterated characteristic sets of the degeneration locus do not allow to decide whether the Dirichlet energy ball is precompact, one may still hope that these difficulties will disappear if the ambient manifold $M$ and the vector fields $X_1,\ldots, X_r$ are restricted to live in a "tame" category not allowing the above pathologies. This is exactly the direction pursued in the next section. 

\section{The o-minimal setting}\label{sec:o_minimal}

It has been known since the 1990s that \emph{o-minimal structures}, originated in model theory, provide a convenient and satisfactory framework for developing "tame topology". By this expression one means a geometric theory that restricts consideration to a class of subsets of $\R^n$, or more general ambient spaces, enjoying all the nice "tameness" properties of semi-algebraic sets (e.g., a well-defined notion of integer dimension, stratifications into smooth cells etc.). See Wilkie's Séminaire Bourbaki \cite{wilkie} for a deeper discussion about the motivations for such a theory. All the results that we need from o-minimality are in van den Dries' book \cite{vandendries}, to which we refer for details. \newline 

Let us recall the notion of an \emph{o-minimal structure} on the ordered field of real numbers. This is a sequence $\mathcal{S}=(\mathcal{S}_n)_{n\geq 1}$ where $\mathcal{S}_n$ is a collection of subsets of $\R^n$, such that the following axioms are satisfied: \begin{enumerate}
\item $\mathcal{S}_n$ is a boolean algebra, that is, it contains the empty set and it is closed under finite unions and complementation, 
\item if $A\in \mathcal{S}_n$ and $B\in \mathcal{S}_m$, then $A\times B\in \mathcal{S}_{n+m}$, 
\item if $A\in \mathcal{S}_{n+1}$, then $\pi(A)\in \mathcal{S}_n$, where $\pi:\R^{n+1}\equiv \R^n\times \R\rightarrow \R^n$ is the projection map, 
\item $\mathcal{S}_n$ contains the semi-algebraic sets, that is, the sets defined by finitely many polynomial identities or inequalities, 
\item $\mathcal{S}_1$ consists precisely of the semi-algebraic sets, i.e., finite unions of points and open intervals. 
	\end{enumerate} 
The elements of $\mathcal{S}_n$ are called the \emph{definable subsets} of $\R^n$ (for the given o-minimal structure $\mathcal{S}$). A function $f:A\rightarrow \R^m$, where $A\subseteq \R^n$, is said to be definable if its graph $\{(x,f(x))\colon\, x\in A\}$ is a definable subset of $\R^{n+m}$. Axioms (1) to (4) are set theoretic reformulations of the notion of definable set in a first order expansion of the ordered field of real numbers. The last axiom is the crucial one, expressing the \emph{o-minimality} of the structure. 

The interest of o-minimal structures stems from two basic and very remarkable facts. The first is that, from the above five axioms alone, it is possible to derive a number of theorems expressing in various ways the non-pathological nature of definable sets. In other words, if no pathology is allowed in dimension one (axiom (5)), then no pathology will occur in higher dimensions either. The second fact is that there is a quite rich supply of o-minimal structures, well beyond semi-algebraic sets (which satisfy the above axioms by the famous Tarski--Seidenberg theorem).

 For example, by \cite{dries_macintyre_marker} (which extends a seminal work of Wilkie, see the introduction of \cite{dries_macintyre_marker}) there is an o-minimal structure for which the following two statements are true. \begin{enumerate}
\item[(i)] Every restricted analytic function is definable. A \emph{restricted analytic function} is the restriction to a closed cube of a function analytic in a neighborhood of it.   
\item[(ii)] The exponential function $\exp:\R\rightarrow \R$ is definable.	
	\end{enumerate}
To appreciate the power of (ii), one should realize that a global analytic function $f:\R^n\rightarrow \R$ cannot be expected to be definable in an o-minimal structure in general. E.g., the zero set of $\sin(x)$, which is definable in a structure where the function $\sin(x)$ is definable, is infinite and would violate axiom (5). Having the global exponential function at our disposal we may construct, by means of axioms (1) to (4), infinitely flat functions (e.g., $e^{-\frac{1}{x^2}}$) and hence various "bump functions". A geometry where (certain) bump functions are allowable is in a sense pretty far from the rigidity of, say, real analytic geometry. Notice that all the examples considered in Section \ref{sec:degeneration} are definable in a structure satisfying (ii). \newline 

Fix an o-minimal structure $\mathcal{S}$ as above. From now on, the word "definable" will refer to $\mathcal{S}$. Following \cite[p.109]{vandendries_article}, we may formulate the concept of definable smooth manifold (without boundary) by introducing an appropriate notion of atlas. Let $M$ be a topological space. An $n$-dimensional \emph{definable smooth atlas} on $M$ is a finite $n$-dimensional smooth atlas $\{(U_i,\phi_i)\}_{i=1}^N$ with the property that each open set $\phi_i(U_i)\subseteq \R^n$ is definable and each transition map $\phi_i(U_i\cap U_j)\rightarrow \phi_j(U_i\cap U_j)$ is definable (in particular, both $\phi_i(U_i\cap U_j)$ and $\phi_j(U_i\cap U_j)$ are definable). A \emph{definable smooth manifold} is then a Hausdorff topological space equipped with an equivalence class of smooth definable atlases (with respect to the usual compatibility relation). Of course, any definable smooth manifold is canonically a smooth manifold. A subset $A$ of a definable smooth manifold is said to be definable if $\phi_i(A\cap U_i)$ is definable for every chart $(U_i,\phi_i)$ in a definable atlas. Similarly, one may define definable maps between definable manifolds. \newline 

We can now prepare the stage for our theorem on Dirichlet integrals in the o-minimal setting. \emph{Let $M$ be a compact definable smooth manifold (without boundary) and let $X_1,\ldots, X_r$ be smooth and definable vector fields on $M$.} By "definable vector field" we mean a vector field which, when represented in a system of local coordinates coming from a definable atlas, has definable coefficients (more elegantly, a vector field which is definable when viewed as a map $M\rightarrow TM$, where the tangent bundle has the canonical definable manifold structure). 

As the degeneration locus $Z$ introduced in Section \ref{sec:degeneration} is an infinitary object, in the present setting we consider a tame analogue of it, for which we can actually prove definability. The \emph{$s$-step degeneration locus} is the set \begin{eqnarray*}
Z^{(s)}&:=&\{p\in M\colon\, \textrm{iterated commutators of length $\leq s$ do not span $T_pM$}\}\\
&=&\{p\in M\colon\,\mathrm{span}_{1\leq i_1,\ldots, i_k\leq r,\, k\leq s}\{[X_{i_1}, \ldots, [X_{i_{k-1}}, X_{i_k}]\ldots](p)\}\neq T_pM\}, 
\end{eqnarray*} where $s\geq 1$. To see that $Z^{(s)}$ is definable, we notice that: \begin{enumerate}
\item Each iterated commutator $X_I:=[X_{i_1}, \ldots, [X_{i_{k-1}}, X_{i_k}]]$, where $I=(i_1, \ldots, i_k)\in \{1,\ldots, r\}^k$ and $k\geq 1$, is definable (because the partial derivatives of a definable function are definable). 
\item  The complement of $Z^{(s)}$ is defined by the first order formula \[
\forall v\in T_pM\quad \exists (\lambda_I)_{I\in \cup_{k\leq s}\{1,\ldots, r\}^k}\in \R^{\cup_{k\leq s}\{1,\ldots, r\}^k}\colon\, \sum_I\lambda_IX_I(p) = v. 
\] 
	\end{enumerate}
For every $s\geq 1$, we have the associated chain of iterated characteristic sets \begin{equation}\label{s-chain}
Z^{(s)}_0:=Z^{(s)}\supseteq Z^{(s)}_1\supseteq\ldots\supseteq Z^{(s)}_k\supseteq \ldots, 
\end{equation}
where $Z_{k+1}^{(s)}:=\mathrm{char}(Z_k^{(s)})$ for every $k\geq 0$. We can now state the main theorem of this section, which gives a neat geometric characterization of precompactness for systems of definable vector fields whose iterated commutators of length less than a given threshold "generically" span the tangent space to the manifold. 

\begin{thm}[Characterization of precompactness of $\mathcal{B}$ in the tame setting]\label{thm:tame}
	Let $M$ be a compact definable smooth manifold and let $X_1,\ldots, X_r$ be smooth and definable vector fields on $M$. Suppose that there exists $s\geq 1$ such that $Z^{(s)}$ has empty interior. Then the following are equivalent: \begin{enumerate}
	\item the Dirichlet energy ball $\mathcal{B}$ (defined with respect to any smooth volume form) is precompact, 
	\item there is no characteristic submanifold, 
	\item $Z^{(s)}_k$ is empty for some $k<+\infty$ (and then for some $k\leq n+1$). 
	\end{enumerate}
	\end{thm}

The hypothesis that $Z^{(s)}$ has empty interior should be compared with Proposition \ref{prp:large_degeneration_locus}. In fact, our proof below proves that if \emph{the degeneration locus $Z=\cap_{s\geq 1}Z^{(s)}$ is contained in a definable set $E$ with empty interior}, then (1), (2) and \begin{enumerate}
	\item[(3')] $E_k$ is empty for some $k<+\infty$ (and then for some $k\leq n+1$)
	\end{enumerate} are equivalent. Here $E_k$ is the $k$-th iterated characteristic set of $E$.  

An inspection of the proof below shows that the tameness of the manifold and the vector fields is really needed only "in a neighborhood of the degeneration locus". We omit the statement of a more general, and slightly more cumbersome, result formalizing this fact.

\section{Proof of Theorem \ref{thm:sufficient}}\label{sec:sufficient}

The key tool is the following lemma stating that the $L^2$ mass of a function of bounded energy cannot be too concentrated near a nowhere characteristic set.

\begin{lem}\label{lem:flow}
Let $A\subseteq M$ be a closed set without characteristic points, that is, $\mathrm{char}(A)$ is empty. Let $V$ be an open neighborhood of $A$. Then there exists a constant $C>0$ such that the following holds. For every $\epsilon>0$ there exists a neighborhood $U_\epsilon$ of $A$ contained in $V$ such that the following inequality holds:\[
	\lVert f \rVert_{L^2(U_\epsilon)}^2 \leq \epsilon \mathbf{E}(f)+C \lVert f\rVert_{L^2(V\setminus U_\epsilon)}^2 \qquad \forall f\in C^\infty(M). 
	\]
\end{lem}

Let us first see how to derive Theorem \ref{thm:sufficient} from the lemma. 

\begin{proof}[Proof of Theorem \ref{thm:sufficient}]
	Let $\epsilon>0$. We claim that \emph{there exists an open neighborhood $U_\epsilon$ of $Z$ and a constant $C_1$ independent of $\epsilon$ such that}  \begin{equation}\label{eq:key_lemma_claim}
		\lVert f\rVert^2_{L^2(U_\epsilon)}\leq \epsilon\mathbf{E}(f)+C_1\lVert f\rVert^2_{L^2(M\setminus U_\epsilon)}\qquad \forall f\in C^\infty(M). 
	\end{equation}
We first derive the thesis from the claim and then turn to its proof. All the estimates below are valid for every $f\in C^\infty(M)$. \newline

Fix $\epsilon>0$. Let $\zeta_\epsilon$ be a test function supported on $\Omega=M\setminus Z$ and identically one on $M\setminus U_\epsilon$, where $U_\epsilon$ is as in the claim. By a standard interpolation inequality for Sobolev norms,  \begin{equation}\label{eq:interpolation}
\lVert f\rVert^2_{L^2(M\setminus U_\epsilon)}	\leq \delta \lVert \zeta_\epsilon f\rVert_{W^s}^2+C_2(\delta, \epsilon)\lVert f\rVert_{W^{-1}}^2,
\end{equation} where $\delta>0$ is arbitrary and $s=s(\epsilon)>0$ is chosen so that we have the subelliptic estimate \begin{equation}\label{eq:subelliptic_2}
	\lVert \zeta_\epsilon f\rVert_{W^s}^2\leq C_3(\epsilon)\left(\mathbf{E}(f)+\lVert f\rVert_{W^{-1}}^2\right). 
	\end{equation}
This is a localized version of \eqref{eq:subelliptic}, which holds because the vector fields $X_1,\ldots, X_r$ satisfy the Hörmander condition on a neighborhood of the support of $\zeta_\epsilon$. Since $s$ depends on the length of the iterated commutators needed to span the tangent space on the region of interest, it ultimately depends on $\epsilon$. Putting together \eqref{eq:key_lemma_claim}, \eqref{eq:interpolation} and \eqref{eq:subelliptic_2}, we get \[
\lVert f\rVert_{L^2(M)}^2\leq (\epsilon+C_4(\epsilon)\delta)\mathbf{E}(f)+C_5(\delta, \epsilon)\lVert f\rVert_{W^{-1}}^2.
\]
Choosing $\delta$ sufficiently small, we finally get \[
\lVert f\rVert_{L^2(M)}^2\leq 2\epsilon\mathbf{E}(f)+C_6(\epsilon)\lVert f\rVert_{W^{-1}}^2.
\]
Since $\epsilon>0$ is arbitrary, this readily implies the desired precompactness. See, e.g., Lemma 1.1 of \cite{kohn_nirenberg}. \newline 

We are left with the proof of the claimed inequality \eqref{eq:key_lemma_claim}. We are going to prove, by backward induction on $\ell$, the following statement. \emph{Let $\epsilon>0$. Then there exists an open neighborhood $U_{\ell, \epsilon}$ of $Z_\ell$ and a constant $B_\ell>0$ \emph{independent of $\epsilon>0$} such that} \begin{equation}\label{key_lemma_claim_ell}
		\lVert f\rVert^2_{L^2(U_{\ell, \epsilon})}\leq \epsilon\mathbf{E}(f)+B_\ell\lVert f\rVert^2_{L^2(M\setminus U_{\ell, \epsilon})}. 
	\end{equation}
	The statement is trivial for $\ell=k$ and coincides with \eqref{eq:key_lemma_claim} for $\ell=0$. Assume that \eqref{key_lemma_claim_ell} holds for $\ell+1$. Let $V$ be an open neighborhood of $Z_{\ell+1}$ such that $\overline{V}\subseteq U_{\ell+1, \epsilon}$. Since $Z_\ell\setminus U_{\ell+1, \epsilon}$ is compact and has no characteristic points, Lemma \ref{lem:flow} gives a neighborhood $W\subseteq M\setminus \overline{V}$ of $Z_\ell\setminus U_{\ell+1, \epsilon}$ and a constant $C$ (independent of $\epsilon$) such that \[
	\lVert f\rVert_{L^2(W)}^2\leq \epsilon \mathbf{E}(f)+C\lVert f\rVert^2_{L^2(M\setminus \left(V\cup W\right))}. 
	\]
	This, together with the inductive assumption \[
	\lVert f\rVert^2_{L^2(U_{\ell+1,\epsilon})}\leq \epsilon\mathbf{E}(f)+B_{\ell+1}\lVert f\rVert^2_{L^2(M\setminus U_{\ell+1,\epsilon})}, 
	\]
	gives \eqref{key_lemma_claim_ell} with $U_{\ell,\epsilon}:= V\cup W$ (modulo a reparametrization in $\epsilon$), as we wanted. 
\end{proof}

\begin{proof}[Proof of Lemma \ref{lem:flow}]

Fix $p\in A$. By assumption, for some $j$ the vector $X_j(p)$ is not $C^1$-Zariski tangent to $A$. This means that there exists a $C^1$ function $u$ that vanishes on $A$ and such that $X_ju(p)\neq 0$. Thus, there is a neighborhood $U$ of $p$ such that $S=U\cap \{u=0\}$ is a $C^1$ hypersurface such that $U\cap A\subseteq S$. 

Let $\Phi(y_1,\ldots, y_{n-1})$ be a local parametrization of $S$ near $p=\Phi(0)$. Since the vector field $X_j$ is transversal to $S$ at $p$, the mapping \[
\Psi:(y_1,\ldots, y_{n-1}, y_n)\longmapsto  \mathrm{exp}_{y_nX}(\Phi(y_1,\ldots, y_{n-1}))
\] is a $C^1$-diffeomorphism of a neighborhood of $0$ in $\R^n$ onto an open neighborhood of $p$ in $M$. (We are denoting by $\exp_{tX}$ the flow at time $t$ of the vector field $X$.) We take the neighborhood in $\R^n$ to be of the form $D\times (-\delta, \delta)$, where $D$ is a bounded neighborhood of $0$ in $\R^{n-1}$ and $\delta>0$. We may assume that $\Psi$ is a diffeomorphism on a larger open set, so that the modulus of the Jacobian of $\Psi$ is bounded away from zero and infinity on $D\times (-\delta, \delta)$. Let $V(\epsilon)=\Psi\left(D\times (-\epsilon,\epsilon)\right)$, where $\epsilon\in (0, \delta)$. Finally, we may also assume that $\overline{V(\delta)}\subseteq V$, the neighborhood appearing in the statement, and $\overline{V(\delta)\setminus V(\epsilon)}$ is disjoint from $A$ for every $\epsilon>0$.  

All of the above works for every fixed $p\in A$. Since $A$ is compact, there are finitely many points $p_1,\ldots, p_N$ and $\delta>0$ with the property that $\bigcup_{k=1}^N V_k(\epsilon)$ contains $A$ for every $\epsilon\in (0,\delta)$, where $V_k(\delta)$ is the open set described in the previous paragraph relative to the point $p_k$. We use in a similar way the symbols $X_{j(k)}$, $D_k$ and $\Psi_k$. Let $f\in C^\infty(M)$. By construction, we have the following chain of inequalities, where $I\lesssim J$ means that $I\leq CJ$ with $C$ independent of $\epsilon\in (0,\delta/3)$:
\begin{eqnarray*}
	&&\lVert f\rVert_{L^2(V_k(\epsilon))}^2 \\
	&\lesssim& \int_{D_k\times (-\epsilon,\epsilon)} |f\circ\Psi_j(y',y_n)|^2\, dy'\, dy_n\qquad (y'=(y_1,\ldots, y_{n-1}))\\
	&\lesssim & \int_{D_k\times (-\epsilon,\epsilon)} \left(|f\circ\Psi_k(y',y_n+2\epsilon)|^2 + \epsilon \int_0^{2\epsilon} |X_{j(k)}f\circ \Psi_k(y',y_n+t)|^2\, dt\right)\, dy'\, dy_n\\
	&\lesssim & \int_{D_k\times (\epsilon,3\epsilon)} |f\circ\Psi_k(y',y_n)|^2 \, dy'\, dy_n + \epsilon^2 \int_{D_k\times (-\epsilon,3\epsilon)}  |X_{j(k)}f\circ \Psi_k(y',y_n+t)|^2\, dy'\, dy_n\\
	&\lesssim & \lVert f\rVert^2_{L^2(V_k(3\epsilon)\setminus V_k(\epsilon))} + \epsilon^2 \lVert X_{j(k)}f\rVert_{L^2(V_k(3\epsilon))}^2\\
	&\leq& \lVert f\rVert^2_{L^2(V_k(3\epsilon)\setminus V_k(\epsilon))} +\epsilon^2\mathbf{E}(f).
\end{eqnarray*}
Summing over $k$ we get \[
\lVert f\rVert_{L^2(\bigcup_{k=1}^NV_k(\epsilon))}^2\lesssim \lVert f\rVert^2_{L^2(\bigcup_{k=1}^NV_k(3\epsilon)\setminus V_k(\epsilon))} + \epsilon^2\mathbf{E}(f).
\]
Notice that the closure $L(\epsilon)$ of $\bigcup_{k=1}^NV_k(3\epsilon)\setminus V_k(\epsilon)$ is compact, contained in $V$, and disjoint from $A$. Setting $U_\epsilon=\bigcup_{k=1}^N V_k(\epsilon)\setminus L(\epsilon)$, we obtain the desired inequality (modulo a reparametrization in $\epsilon$). 

\end{proof}

\section{Characteristic sets in the o-minimal setting}\label{sec:proof_tame_1}

In the last three sections of the paper we prove Theorem \ref{thm:tame}. Fix an o-minimal structure. Let $X_1,\ldots, X_r:\R^n\rightarrow \R^n$ be definable vector fields (not necessarily continuous). If $A\subseteq \R^n$, the set $\mathrm{char}(A)$ of characteristic points of $A$ is well-defined. In fact, no regularity of the vector fields is required for Definition \ref{dfn:char} to make sense (in this generality, the set of characteristic points of a closed set need not be closed). The goal of this section is to prove the following lemma, reassuring us that by taking characteristic points we never leave the tame category. 

\begin{lem}\label{lem:def_char}
	If $A$ is definable, then $\mathrm{char}(A)$ is definable. 	
\end{lem}

Its proof is in turn based on a lemma expressing the definability of the $C^1$-Zariski tangent distribution to a definable set. If $A\subseteq \R^n$ is a set, then we denote by $T_xA$ the set of its $C^1$-Zariski tangent vectors. 

\begin{lem}\label{lem:def_tangent} If $A\subseteq \R^n$ is closed and definable, then the set \[
	TA = \{(x,v)\in \R^n \colon\, x\in A, \, v\in T_xA\}
	\] 
	is definable. 
	\end{lem}

Our proof of Lemma \ref{lem:def_tangent} does not rely on o-minimality and works on every expansion of the ordered field of real numbers (i.e., for any structure satisfying axioms (1)-(4) in Section \ref{sec:o_minimal}), a fact that may be of independent interest. A simpler proof exploiting o-minimality may well exist. 

Deriving Lemma \ref{lem:def_char} from Lemma \ref{lem:def_tangent} is an easy task. 

\begin{proof}[Proof of Lemma \ref{lem:def_char}] 
	Since $A\cap \mathrm{char}(\overline{A})=\mathrm{char}(A)$, we may assume without loss of generality that $A$ is closed. Letting \[
	E_j = \{(x,X_j(x))\in \R^{2n}\colon\, x\in A\}\qquad (j=1,\ldots, r),  
	\]
	we have \[
	\mathrm{char}(A) = \cap_{j=1}^r\pi\left(TA\cap E_j\right), 
	\]
	where $\pi:\R^{2n}\rightarrow \R^n$ is the projection onto the first $n$ coordinates. Each $E_j$ is definable, because the vector fields are definable, and $TA$ is definable by Lemma \ref{lem:def_tangent}. Since $\mathrm{char}(A)$ is built out of $TA$ and $E_1,\ldots, E_r$ via boolean operations and linear projections, we get the thesis.
	\end{proof}

The proof of Lemma \ref{lem:def_tangent} is more complicated. To understand why, notice that unraveling the definition of Zariski tangency, we see that $TA$ is defined by the formula (in $(x,v)\in \R^{2n}$) \[
\forall f\in C^1(\R^n) \left((\forall x \in A\, f(x)=0) \quad \Longrightarrow \quad \sum_{k=1}^n\partial_{x_k}f(x) v_k=0\right).
\]
This is not a first order formula in the language of our structure, and hence it is useless for settling the definability of $TA$. We will overcome this difficulty by finding a first-order description of $TA$. In order to do this, we exploit some machinery from the theory surrounding the Whitney extension problem. \newline 

We need to describe a couple of operations that can be performed on subsets of $\R^n\times \R^n$. We think of a subset $B\subseteq \R^n\times \R^n$ as a fibered collection of subsets of $\R^n$: for each $x\in \R^n$ we denote by $B_x$ the \emph{fiber of $B$ over $x$}, that is, the set \[
B_x:=\{v\in\R^n\colon\, (x,v)\in \R^n\}. 
\]

The first operation is \emph{fiberwise linearization}. Given $B\subseteq \R^n\times \R^n$, it produces the new set $\widetilde{B}$ whose fiber $\widetilde{B}_x$ is the linear span of the original fiber $B_x$. By elementary linear algebra, $\widetilde{B}$ is defined by the first order formula in $(x,v)$ \[
\exists \lambda_1,\ldots, \lambda_n\in \R,\, \exists v_1,\ldots, v_n\in \R^n\colon\quad (x,v_1), \ldots, (x,v_n)\in B \, \text{and}\, v=\sum_{k=1}^n\lambda_kv_k.
\]
Thus, if $B$ is definable, the same is true of its fiberwise linearization $\widetilde{B}$. 

The second operation is simply topological \emph{closure} as a subset of $\R^n\times \R^n$. We denote it, as usual, by $\overline{B}$. The closure of a definable set is definable, so this second operation preserves definability too. 

It is convenient to introduce the symbol $\Lambda(B)$ for the result of the operation of closure followed by fiberwise linearization, i.e., $\Lambda(B):=\widetilde{\overline{B}}$. \newline 

Next, we need to recall the notion of \emph{paratangent cone} of a set $A\subseteq \R^n$. This is a subset $\mathrm{ptg}(A)\subseteq \R^n\times \R^n$, whose fiber over $x$ consists of vectors $v$ that are limits of secants of $A$, that is, $v\in \mathrm{ptg}(A)_x$ if and only if there exist $\sigma\in \R$ and sequences $\{p_j\}_j$, $\{q_j\}_j$ in $A$ such that: \begin{enumerate} 
	\item $\lim_{j\rightarrow +\infty}p_j=\lim_{j\rightarrow +\infty}q_j=x$, 
	\item $p_j\neq q_j$ for every $j$, 
	\item $v=\sigma \lim_{j\rightarrow +\infty} \frac{p_j-q_j}{|p_j-q_j|}$. 
\end{enumerate}
Replacing the limits with an $\epsilon-\delta$ definition and observing that $\frac{x}{|x|}$ is a semialgebraic function, we see that if the set $A$ is definable, then $\mathrm{ptg}(A)$ is also definable. \newline 

We can finally formulate the following deep theorem of Glaeser. See \cite[Chapter 2, Proposition VII]{glaeser} or \cite[Theorem 4.1(3)]{tierno}. 

\begin{thm}
Let $A\subseteq \R^n$ be an arbitrary closed set. Then \[
TA = \underbrace{\Lambda\circ \cdots \circ \Lambda}_{\textrm{$2n$ times}} (\widetilde{\mathrm{ptg}(A)}). 
\] 	
	\end{thm}

\begin{proof}[Proof of Lemma \ref{lem:def_tangent}]
By Glaeser's theorem, $TA$ is obtained by applying to $A$ finitely many operations, each of which preserves definability. 
	\end{proof}

\section{Finite stabilization lemma}\label{sec:proof_tame_2}

We continue with the setting of Section \ref{sec:proof_tame_1}, that is, $X_1,\ldots, X_r:\R^n\rightarrow \R^n$ are definable vector fields (not necessarily continuous), and characteristic sets are defined with respect to them. In this section we prove the following crucial result. 

\begin{lem}[Finite stabilization]\label{lem:finite_termination} Let $E\subseteq \R^n$ be closed and definable and let \[
	E_0\supseteq E_1\supseteq\ldots\supseteq E_k\supseteq \ldots
	\]
	be the corresponding chain of iterated characteristic sets, that is, $E_0:=E$ and $E_{k+1}:=\mathrm{char}(E_k)$ for every $k\geq 0$. Then there exists $k\leq n+1$ such that $E_k=E_{k+1}=\ldots$. 
\end{lem}

The proof of Lemma \ref{lem:finite_termination} is based on a basic stratification theorem for definable sets in o-minimal structures. In order to formulate it, we need to recall the notion of \emph{cell} from \cite[p.~50]{vandendries}. Cells are special definable subsets of $\R^n$, each having a \emph{dimension} $k\in \{0,1,\ldots, n\}$. They are defined recursively as follows: 
 \begin{enumerate}
\item a $0$-dimensional cell in $\R^1$ is a point, 
\item a $1$-dimensional cell in $\R^1$ is an open interval, 
\item if $C'$ is a $k$-dimensional cell in $\R^{n-1}$ (in this case $k<n$) and $f:C'\rightarrow \R$ is a continuous definable function, then \[
C=\{(x',f(x))\in \R^n\colon\, x'\in C'\}
\] is a $k$-dimensional cell in $\R^n$, 
\item if $C'$ is a $(k-1)$-dimensional cell in $\R^{n-1}$ and $f,g:C'\rightarrow \R$ are continuous definable functions, then \[
C=\{(x',y)\in \R^n\colon\, x'\in C', \, f(x')<y<g(x')\}
\] is a $k$-dimensional cell in $\R^n$, and the same holds if $f\equiv-\infty$ or $g\equiv+\infty$. 
	\end{enumerate}
For later purposes, we recall that the notion of \emph{$C^1$ cell} is defined analogously, replacing the continuity requirement on $f$ and $g$ with $C^1$ regularity (when these are real-valued). Notice that $k$-dimensional cells are topological submanifolds of dimension $k$, and $k$-dimensional $C^1$ cells are $C^1$ submanifolds of dimension $k$.\newline 

We adopt the terminology of \cite{vandendries} and call \emph{frontier} of a set $A\subseteq \R^n$ the set $\overline{A}\setminus A$. 
Let $A$ be a closed definable set. A \emph{stratification} of $A$ is a partition of $A$ into finitely many cells, called \emph{strata}, such that the frontier of each stratum is a union of lower-dimensional strata. The next theorem is Proposition 1.13 in \cite[p.68]{vandendries}. 

\begin{thm}[Stratification of closed definable sets]\label{thm:stratification}
	Let $A$ be a closed definable set and let $A_1,\ldots, A_m$ be definable subsets of $A$. Then there exists a stratification of $A$ partitioning each of $A_1,\ldots, A_m$ (i.e., each of $A_1,\ldots, A_m$ is a union of strata). 
\end{thm}

\begin{proof}[Proof of Lemma \ref{lem:finite_termination}] By Lemma \ref{lem:def_char}, every $E_j$ is definable. Applying iteratively Theorem \ref{thm:stratification}, we may find, for each $j\geq 0$, a stratification $\mathcal{P}_j$ of $E_j$ with the property that each stratum of $\mathcal{P}_{j+1}$ is contained in a stratum of $\mathcal{P}_j$.  
	
We claim that \emph{if $C$ is a cell of $\mathcal{P}_k$ of dimension at least $n-k+1$, then $C\subseteq E_j$ for every $j\geq k$}. The case $k=0$ is vacuously true and the case $k=n+1$ is the thesis. 
	
We are going to use the following elementary observation, which has nothing to do with o-minimality: \emph{if $A,B\subseteq \R^n$ are arbitrary subsets such that $A\subseteq \mathrm{char}(B)$ and $A$ is open in $B$, then $A=\mathrm{char}(A)$}. 

Let $m\geq 1$ and assume the validity of the claim for every $k<m$. Let $C$ be a stratum of $\mathcal{P}_m$ of dimension at least $n-m+1$, and let $C'$ be the stratum of $\mathcal{P}_{m-1}$ that contains it. 

If $C'$ is contained in the frontier of a stratum $C''$ of $\mathcal{P}_{m-1}$ of dimension at least $n-m+2$, then by the inductive hypothesis $C''\subseteq E_j$ for every $j\geq m-1$. Since $E_j$ is closed, $C\subseteq C'\subseteq E_j$ for every $j\geq m$ and we are done. 

If $C'$ is not contained in the closure of a stratum of dimension at least $n-m+2$ of $\mathcal{P}_{m-1}$, then it must have the same dimension as $C$ and be open in $E_{m-1}$. These two facts imply that $C$ itself is open in $E_{m-1}$. Since $C\subseteq E_m=\mathrm{char}(E_{m-1})$, by the elementary observation we have $C=\mathrm{char}(C)$, which in turn gives $C\subseteq E_j$ for every $j\geq m$, as we wanted. \end{proof}

\section{Proof of Theorem \ref{thm:tame}}\label{sec:proof_tame_3}

Putting together the results of the previous two sections, we may now give the proof of Theorem \ref{thm:tame}. More precisely, we prove the slight generalization discussed after the statement of Theorem \ref{thm:tame}. 

First of all, condition (1) implies condition (2) by Proposition \ref{prp:weierstrass}. 

Let $E$ be a definable set with empty interior that contains the degeneration locus. Since the closure of a definable set with empty interior has empty interior, we may assume without loss of generality that $E$ is closed. Because $E\supseteq Z$, condition (3') (see the remarks after the statement of the theorem) implies that $Z_k$ is empty for some $k<+\infty$, and hence (1) holds by Theorem \ref{thm:sufficient}.

We are left with the proof that (2) implies (3'). Assume (2). By Lemma \ref{lem:finite_termination}, $E_{n+1}=E_{n+2}=\ldots$ and we have to prove that $E_{n+1}$ is empty. We argue by contradiction, assuming it non-empty. Let $p\in E_{n+1}$ and let $(U,\phi)$ be a chart in a definable atlas such that $p\in U$ and  $\phi(U)=\R^n$. To achieve this, one may first restrict the domain of the chart so that $\phi(U)$ is a ball, and then compose $\phi$ with a semi-algebraic diffeomorphism of the ball onto $\R^n$. The notion of characteristic point is local (i.e., $\mathrm{char}(E)\cap U$ is the same as the set of characteristic points of $E\cap U$ viewed as a subset of the manifold $U$) and diffeomorphism invariant. We omit the straightforward proofs of this facts. Thus, $E':=\phi(E\cap U)$ is a closed definable subset of $\R^n$ such that the iterated characteristic set $E'_{n+1}$, computed with respect to the pushed-forward definable vector fields $\phi_*X_j$, equals $E'_{n+2}$ and is non-empty. By Theorem \ref{thm:stratification}, $E'_{n+1}$ admits a stratification consisting of at least one cell. Any top-dimensional cell is relatively open in $E'_{n+1}$. By the $C^1$ decomposition theorem (see \cite[p. 115]{vandendries}), we may further partition such a cell into $C^1$ cells, at least one of which must be of the same dimension, and hence also relatively open in $E'_{n+1}$. This proves that there is an open set $V\subseteq \R^n$ such that $N=V\cap E'_{n+1}$ is a $C^1$ submanifold. Since $E'_{n+1}=E'_{n+2}$, $N$ is characteristic with respect to the pushed-forward vector fields. Notice that $N$ has positive codimension, because $E$ has empty interior. The same is true of $\phi^{-1}(N)$. By Proposition \ref{prp:weierstrass}, the Dirichlet energy ball is not precompact. This is the desired contradiction. The proof is complete. 

\appendix
\section{Proof of Proposition \ref{prp:weierstrass}}\label{sec:weierstrass}

We work in local coordinates $(x_1,\ldots, x_n)$ centered at a point of $N$ and such that (a local piece of) $N$ is defined by the equations \[x_1=\ldots=x_s=0, \qquad s\geq 1. \]

If $X_j=\sum_{k=1}^n b_{jk}(x)\partial_{x_k}$ ($j=1,\ldots, r$), then the fact that $N$ is characteristic is expressed by the vanishing \[
b_{jk}(0,\ldots, 0, x_{s+1}, \ldots, x_n)=0\qquad \forall k=1,\ldots, s, \, \forall j.\]
Thus, there are smooth functions $b_{jk}^\ell$ ($k,\ell=1,\ldots, s$) such that
\begin{equation}\label{weierstrass_char}
b_{jk}(x) = \sum_{\ell=1}^s x_\ell b_{jk\ell}(x)\qquad \forall k=1,\ldots, s, \, \forall j. 
\end{equation}
Let $\eta$ be a test function supported near the origin of $\R^n$ and define the $L^2$-rescaled function \[
\eta_t(x)=t^{\frac{s}{2}}\eta(tx_1,\ldots, tx_s, x_{s+1}, \ldots, x_n), 
\]
which we view as a function on $M$, via the local coordinates, when $t$ is large. It is easy to see that $\lVert \eta_t\rVert_{L^2(M)}^2\sim_{t\rightarrow +\infty}\int \eta(x)^2\rho(0,\ldots, 0, x_{s+1}, \ldots, x_n)\, dx_1\ldots dx_n$, where $\rho$ is the density of the volume form in local coordinates. We claim that 
\begin{equation}\label{weierstrass_dirichlet}
	\mathbf{E}(\eta_t)=O_{t\rightarrow +\infty}(1). 
\end{equation}
From this, Proposition \ref{prp:weierstrass} follows readily, because $\{\eta_t\colon\, t\geq T\}$ is not precompact in $L^2(M)$ for any $T<+\infty$. 
To prove \eqref{weierstrass_dirichlet}, observe that $\partial_{x_k}\eta_t= t^\epsilon(\partial_{x_k}\eta)_t$ with $\epsilon=1$ if $k\leq s$ and $\epsilon=0$ otherwise. Thus, by \eqref{weierstrass_char},  \begin{eqnarray*}
	|X_j\eta_t|^2 &=& \sum_{k,k'\leq s} \sum_{\ell, \ell'\leq s}t^2x_\ell x_{\ell'} b_{jk\ell} b_{jk'\ell'} (\partial_{x_k}\eta)_t (\partial_{x_{k'}}\eta)_t\\
	&+& 2\sum_{k\leq s}\sum_{k'\geq s+1}\sum_{\ell \leq s} tx_\ell b_{jk\ell}b_{jk'}(\partial_{x_k}\eta)_t (\partial_{x_{k'}}\eta)_t
	\\
	&+& \sum_{k,k'\geq s+1} b_{jk}b_{jk'} (\partial_{x_k}\eta)_t (\partial_{x_{k'}}\eta)_t, 
\end{eqnarray*}
which is a finite sum of terms of the form $t^s f(tx_1,\ldots, tx_s, x_{s+1}, \ldots,x_n)g(x)$, where $f$ is compactly supported and $g$ is smooth. It follows immediately that \[
\mathbf{E}(\eta_t)=\sum_{j=1}^r\int 	|X_j\eta_t|^2 \, dm = O_{t\rightarrow +\infty}(1), 
\] as desired. 

\section{Implications for global regularity of sum-of-squares operators}\label{sec:hypoellipticity}

In this appendix we discuss the connection between precompactness of the Dirichlet energy ball and global regularity properties of an associated sum-of-squares operator, which follows from classical work of Kohn and Nirenberg. Based on this, we make some comments on a result of K. Amano related to our Theorem \ref{thm:sufficient}. \newline 

Let $M$ be a compact manifold without boundary and $X_1,\ldots, X_r$ smooth vector fields on $M$. The non-negative quadratic form \eqref{eq:dirichlet_energy}, defined on its natural domain \[D=\{f\in L^2(M)\colon\, X_jf\in L^2(M)\quad \forall j\},\] is  easily seen to be closed, that is, $D$ is a Hilbert space with respect to the norm $\mathbf{E}(f)+\lVert f\rVert_{L^2(M)}^2$. Thus, $\mathbf{E}$ is the quadratic form of a non-negative self-adjoint operator $\mathcal{L}$, defined as follows: $\mathcal{L}f=g$ if and only if $f\in D$, i.e., it has finite Dirichlet energy, and \[
\int_M g \varphi\, dm = \sum_{j=1}^r\int_M X_jfX_j\varphi\, dm\qquad \forall \varphi\in D.
\]
 By Friedrichs Lemma (see, e.g., \cite[Lemma 25.4]{treves}), $C^\infty(M)$ is a core for the Dirichlet energy, and it is enough to test the above identity on $C^\infty(M)$. Thus, an equivalent definition of $\mathcal{L}$ is the following: $\mathcal{L}f=g$ if and only if $f$ has finite Dirichlet energy, and $\sum_{j=1}^rX_j^\dagger X_jf=g$ in the sense of distributions, where the dagger notation indicates formal adjunction with respect to the background volume form. The partial differential operator with smooth coefficients \begin{equation}\label{sum_of_squares}
 	P=\sum_{j=1}^r X_j^\dagger X_j\end{equation} is a \emph{sum-of-squares operator}. 

By Theorem 3 of \cite{kohn_nirenberg}, if the Dirichlet energy ball $\mathcal{B}$ is precompact, then every distributional solution with finite Dirichlet energy of the equation \[
Pu=f, 
\] with a smooth datum $f$, is itself smooth. Combining this with Theorem \ref{thm:sufficient} above, we obtain the following. 

\begin{cor}\label{cor:glob_hypo} If there exists a natural number $k$ with the property that $Z_k$ is empty, then the sum-of-squares operator \eqref{sum_of_squares} satisfies the following global regularity property: if $u$ has finite Dirichlet energy and $Pu$ is everywhere smooth, then $u$ itself is everywhere smooth. 
\end{cor}

We remark that a variant of the above corollary holds on arbitrary (not necessarily compact) manifolds, if \emph{the degeneration locus $Z$ is compact} and the global regularity property is modified as follows: \emph{if $u\in \mathcal{D}'(M)$ has locally finite Dirichlet energy in a neighborhood of $Z$, that is, $u$ and each $X_ju$ are square-integrable in some neighborhood of $Z$, and $Pu$ is everywhere smooth, then $u$ is smooth}. Let us sketch a proof of this generalization. Since $Z$ is compact, it is always possible to modify the vector fields $X_j$ and the ambient manifold $M$ away from $Z$ in such a way that the new manifold $\widetilde{M}$ is compact and the new vector fields $\widetilde{X_j}$ span the tangent space outside a neighborhood of $Z$. Since the original sum-of-squares operator $P$ is hypoelliptic outside $Z$ by Hörmander's theorem \cite{hormander}, any distributional solution $u$ of $Pu=f$ with $f$ smooth may be modified away from $Z$ in such a way that the new distribution $\widetilde{u}$ is a solution of a new equation $\widetilde{P}\widetilde{u}=\widetilde{f}$ with $\widetilde{f}$ smooth and $\widetilde{P}$ a sum-of-squares operator associated to the new vector fields $\widetilde{X_j}$. The corollary above then implies that $\widetilde{u}$, and hence $u$, is smooth in a neighborhood of $Z$, as we wanted.  \newline 

Notice that one cannot derive local hypoellipticity from the assumptions of Corollary \ref{cor:glob_hypo}. In fact, it has been shown by Kusuoka and Stroock that the sum-of-squares operator associated to the following vector fields on $M=\R^3_{x_1,x_2,x_3}$: \[
X_1=\partial_{x_1}, \quad X_2 = e^{-\frac{1}{|x_1|}}\partial_{x_2}, \quad X_3=\partial_{x_3}
\]
is not locally hypoelliptic (see \cite[Theorem 8.41]{kusuoka_stroock}). It is easy to see that $Z_1$ is empty in this case. (In passing, we observe that the above vector fields are definable in an appropriate o-minimal structure). \newline 

We now comment on K. Amano's work \cite{amano} on global hypoellipticity for second order linear PDEs. The setting of \cite{amano} is an open set $\Omega\subseteq \R^n$. We recall that the linear partial differential operator $P$ is said to be \emph{globally hypoelliptic} if any globally defined distribution $u\in \mathcal{D}'(\Omega)$, such that $Pu\in C^\infty(\Omega)$, is automatically smooth. When specialized to sum-of-squares operators of the form \eqref{sum_of_squares}, Theorem 1 in Amano's paper states that the following is a sufficient condition for global hypoellipticity: \emph{the degeneration locus $Z\subseteq \Omega$ of the family of vector fields $X_1,\ldots, X_r$ is compact and there exists $\Phi\in C^\infty(\Omega)$ and finitely many vector fields $Y_1,\ldots, Y_N$ in the $C^\infty(\Omega)$-module generated by $X_1, \ldots, X_r$, such that $Z=\{p\in \Omega\colon\, \Phi(p)=0\}$ and}  \[
Z=\bigcup_{k=1}^N\{p\in Z \colon\, Y_k\cdots Y_1\Phi(p)\neq 0\}. 
\]
Notice that this condition generalizes a previous one by Oleinik--Radkevič (see Theorem 2.5.3 of \cite{oleinik_radkevic}). 
Amano's condition implies the eventual emptiness of iterated characteristic sets. In fact, a simple inductive argument shows that \[
Z_k\subseteq \bigcap_{i=1}^k \{p\in Z\colon\, Y_i\cdots Y_1\Phi(p)= 0\}
\]
for every $k\geq 1$, which immediately yields emptiness of $Z_N$. Thus, both the hypothesis and the conclusion in Corollary \ref{cor:glob_hypo} are weaker than in Amano's theorem. We find the assumption of our corollary (and Theorem \ref{thm:sufficient} above) to be conceptually clearer, and it seems plausible that working slightly harder one might derive genuine global hypoellipticity from it.

\end{document}